\documentclass[a4, 11pt]{amsart}
\usepackage[centertags]{amsmath}
\usepackage{amsfonts}
\usepackage{amssymb}
\usepackage{amsthm}
 \newtheorem{thm}{Theorem}[section]
 \newtheorem{cor}[thm]{Corollary}
 
 \newtheorem{prop}[thm]{Proposition}
 \theoremstyle{definition}
 \newtheorem{defn}[thm]{Definition}
 \theoremstyle{remark}
 \newtheorem{rem}[thm]{Remark}
\theoremstyle{example}

 \numberwithin{equation}{section}
\begin{document}
\title{New $r$-Matrices for Lie Bialgebra Structures over Polynomials}
\author{IULIA POP}
\author{JULIA YERMOLOVA--MAGNUSSON}
\address{Department of Mathematical Sciences, University of Gothenburg, 
Sweden. Email: iulia@chalmers.se; md1jm@chalmers.se}

\begin{abstract}
For a finite dimensional simple complex Lie algebra $\mathfrak{g}$, Lie bialgebra structures on $\mathfrak{g}[[u]]$ and $\mathfrak{g}[u]$ were classified by
Montaner, Stolin and Zelmanov. In our paper, we provide an explicit algorithm to produce $r$-matrices which correspond to Lie bialgebra structures over polynomials.
 \end{abstract}
\maketitle
\section{Introduction}

In the recent paper \cite{SZ}, Montaner, Stolin and Zelmanov discussed the classification of Lie bialgebras and corresponding quantum groups over Taylor series and polynomials. We devote the introduction to giving an overview of the main results in \cite{SZ}. 

\subsection{Lie bialgebra structures over Taylor series}
Let $\mathfrak{g}$ denote a finite dimensional simple complex Lie algebra. Consider $\mathfrak{g}[[u]]$ the algebra of Taylor series over $\mathfrak{g}$. In order to classify Lie bialgebra structures on $L:=\mathfrak{g}[[u]]$, one needs some information on the corresponding Drinfeld double. As it is known (see \cite{D}), any Lie bialgebra structure $\mu$ on $L$ induces a Lie algebra structure on the space of restricted functionals $L^*$ on $L$. Moreover, the vector space $D_{\mu}(L):=L\oplus L^*$ can be equipped with a unique Lie algebra structure, which extends the brackets on $L$ and $L^*$ respectively, and is such that the natural nondegenerate form on $D_{\mu}(L)$ is invariant. The Lie algebra $D_{\mu}(L)$ is called the Drinfeld double corresponding to $\mu$. 

In  \cite{SZ} it was proved that $D_{\mu}(L)$ is either a trivial extension of $L$ or has one of the following forms:

I. The Lie algebra $\mathfrak{g}((u))$ of Laurent polynomials over 
$\mathfrak{g}$ with the nondegenerate symmetric bilinear form $Q_1$ defined by 
\[Q_1(xf_1(u),yf_2(u))=K(x,y)\cdot T_1(f_1(u)f_2(u)),\]
for any $x, y\in \mathfrak{g}$, $f_1(u),f_2(u)\in\mathbb{C}((u))$, where $K$ denotes the Killing form on 
$\mathfrak{g}$ and $T_1$ is given by $T_1(u^k)=0$ for all $k\geq 0$, $T_1(u^{-1})=1$ and $T_1(u^{-k-1})=a_k$, for all $k\geq 1$. 

II. The Lie algebra $\mathfrak{g}((u))\oplus \mathfrak{g}$ with the nondegenerate symmetric bilinear form $Q_2$ defined by 
\[Q_2(x_1f_1(u)+x_2,y_1f_2(u)+y_2)=K(x_1,y_1)\cdot T_2(f_1(u)f_2(u))-K(x_2,y_2),\]
for any $x_1, y_1, x_2, y_2 \in \mathfrak{g}$, $f_1(u),f_2(u)\in\mathbb{C}((u))$, and $T_2$ is given by 
$T_2(u^k)=0$ for all $k\geq 1$, $T_2(1)=1$ and $T_2(u^{-k})=a_k$, for all $k\geq 1$. 

III. The Lie algebra $\mathfrak{g}((u))\oplus (\mathfrak{g}+\varepsilon\mathfrak{g})$, where $\varepsilon^2=0$, with the nondegenerate symmetric bilinear form $Q_2$ defined by 
\[Q_3(x_1f_1(u)+x_2+\varepsilon x_3,y_1f_2(u)+y_2+\varepsilon y_3)=\]
\[K(x_1,y_1)\cdot T_3(f_1(u)f_2(u))-K(x_3,y_2)-K(x_2,y_3)-a\cdot K(x_2,y_2),\]
for any $x_1, y_1, x_2, y_2,x_3, y_3\in \mathfrak{g}$, $f_1(u),f_2(u)\in\mathbb{C}((u))$, where $T_3(u^k)=0$ for all $k\geq 2$, $T_3(u)=1$,  $T_3(1)=a$ and 
$T_3(u^{-k})=a_{k+1}$, for all $k\geq 1$. 

In each of the above cases, consider

\[a(u):=1+\sum_{k=1}^{\infty}a_ku^k.\]

Let $\mathrm{Aut}_0(\mathbb{C}[[u]])$ be the group of infinite series 
$\gamma=x+\gamma_2x^2+\gamma_3x^3+...$, with respect to substitution. 
In \cite{SZ} it was shown that up to an automorphism $\gamma\in\mathrm{Aut}_0(\mathbb{C}[[u]])$, one may suppose that $a(u)=1$.

\begin{thm}
\cite{SZ} Let $\mu$ be any Lie bialgebra structure on $\mathfrak{g}[[u]]$. 
Then there exists $\gamma\in\mathrm{Aut}_0(\mathbb{C}[[u]])$ such that 
$D_{\mu}(\mathfrak{g}[[u]])$ is isomorphic, via $\gamma$, to one of the Lie algebras in cases I--III with $a(u)=1$. 
\end{thm}

\subsection{Lie bialgebra structures over polynomials}
Let us now focus on the classification of Lie bialgebra structures on $\mathfrak{g}[u]$. Denote by  
$\delta$ any such structure. 
It was shown in \cite{SZ} that $\delta$ can be naturally extended to a Lie bialgebra structure $\bar{\delta}$ on $\mathfrak{g}[[u]]$, by letting 
\[\bar{\delta}(\sum_{n=0}^{\infty}x_nu^n)=\sum_{n=0}^{\infty}\delta(x_nu^n).\] 
The corresponding Drinfeld double $D_{\bar{\delta}}(\mathfrak{g}[[u]])$ is either a trivial extension of $\mathfrak{g}[[u]]$ or is isomorphic to one of the Lie algebras in cases I--III. 
We note that $a(u)$ is not necessarily constant. An automorphism of $\mathbb{C}[[u]]$ is not necessarily an automorphism of $\mathbb{C}[u]$. Moreover, there are certain restrictions that have to be imposed on $a(u)$, depending on the case.

\textbf{Case I.} 
Let us suppose that $D_{\bar{\delta}}(L)=\mathfrak{g}((u))$.
This implies that there exists a  Lagrangian subalgebra $W$ of 
$\mathfrak{g}((u))$, with bracket induced by  $\bar{\delta}$, such that 
$W\oplus L=\mathfrak{g}((u))$. It was shown that the following properties hold:

(i) $W$ is bounded, i.e., there exists a natural number $n$ such that $W\subseteq u^n\mathfrak{g}[u^{-1}]$.

(ii) $W\cdot\mathbb{C}[[u^{-1}]]$ is an order in $\mathfrak{g}((u^{-1}))$. Here we recall that an order in $\mathfrak{g}((u^{-1}))$ is a Lie subalgebra $V$ for which there exist natural numbers $k$ and $n$ such that $u^{-n}\mathfrak{g}[[u^{-1}]]\subseteq V \subseteq u^k\mathfrak{g}[[u^{-1}]]$.

(iii) Using the general theory of orders \cite{S2}, the authors show that, by means of a gauge transformation $\sigma(u)\in \mathrm{Aut}_{\mathbb{C}[u]}(\mathfrak{g}[u])$, one can embed $W\cdot\mathbb{C}[[u^{-1}]]$ into a special order denoted by $\mathbb{O}_\alpha$, constructed in the following way: 
 
Let $\mathfrak{h}$ be a Cartan subalgebra of $\mathfrak{g}$ with
 the corresponding set of roots $R$ and a choice of simple
 roots $\Gamma$. Denote by $\mathfrak{g}_{\alpha}$ the root space
 corresponding to a root $\alpha$. Let $\mathfrak{h}(\mathbb{R})$
 be the set of all $h\in\mathfrak{h}$ such that
 $\alpha(h)\in\mathbb{R}$ for all $\alpha\in R$.
 Consider the valuation on $\mathbb{C}((u^{-1}))$ defined by
 $v(\sum_{k\geq n}a_{k}u^{-k})=n$. For any root $\alpha$ and any
 $h\in\mathfrak{h}(\mathbb{R})$, set
 $M_{\alpha}(h)$:=$\{f\in\mathbb{C}((u^{-1})):v(f)\geq \alpha(h)\}$.
 Consider
 \[\mathbb{O}_{h}:=\mathfrak{h}[[u^{-1}]]\oplus(\oplus_{\alpha\in R}M_{\alpha}(h)\otimes\mathfrak{g}_{\alpha}).\]

Vertices of the above simplex correspond to vertices of the extended Dynkin diagram of
$\mathfrak{g}$, the correspondence being given by the following rule:
\[0\leftrightarrow\ -\alpha_{\max}\]\[h_{i}\leftrightarrow\alpha_{i}\]
where $\alpha_{i}(h_{j})=\delta_{ij}/k_{j}$ and $k_{j}$ are given by the relation $\sum
k_{j}\alpha_{j}=\alpha_{\max}$. One writes $\mathbb{O}_{\alpha}$ instead of
$\mathbb{O}_{h}$ if $\alpha$ is the root which corresponds to the vertex $h$ 
and $\mathbb{O}_{-\alpha_{\max}}$ instead of $\mathbb{O}_0$.

Then there exists $\sigma(u)\in \mathrm{Aut}_{\mathbb{C}[u]}(\mathfrak{g}[u])$ 
such that $\sigma(u)(W\cdot\mathbb{C}[[u^{-1}]])\subseteq \mathbb{O}_{\alpha}$, where $\alpha$ is either a simple root or $-\alpha_{\max}$. Here one makes the remark that for any $\sigma\in\mathrm{Aut}_{\mathbb{C}[u]}(\mathfrak{g}[u])$, there exists a natural embedding
$$\mathrm{Aut}_{\mathbb{C}[u]}(\mathfrak{g}[u])\hookrightarrow\mathrm{Aut}_{\mathbb{C}((u^{-1}))}(\mathfrak{g}((u^{-1}))),$$
defined by the formula
$\sigma(u^{-k}x)=u^{-k}\sigma(x)$, for any $x\in\mathfrak{g}[u]$; hence one can regard $\sigma(u)$ as acting on $\mathfrak{g}((u^{-1}))$.

\begin{thm}\cite{SZ}\label{condCaseI}
Let $\alpha_{\max}=\sum  k_j\alpha_j$, $\alpha_j\in\Gamma$ and  
$\sigma\in\mathrm{Aut}_{\mathbb{C}[u]}(\mathfrak{g}[u])$.

(1) Assume that $\sigma(u)(W\cdot\mathbb{C}[[u^{-1}]])\subseteq \mathbb{O}_{-\alpha_{\max}}$. Then $\frac{1}{a(u)}$ is a polynomial of degree at most 2. 

(2) Assume that 
$\sigma(u)(W\cdot\mathbb{C}[[u^{-1}]])\subseteq \mathbb{O}_{\alpha_i}$, for some $i$. Then $\frac{1}{a(u)}$ is a polynomial of degree at most 2 if $k_i=1$, and 
$\frac{1}{a(u)}$ is a polynomial of degree at most 1 if $k_i>1$.

\end{thm}

\vspace*{0.5cm}

\textbf{Case II.} 
Let us suppose $\bar{\delta}$ satisfies the condition 
$D_{\bar{\delta}}(L)=\mathfrak{g}((u))\oplus \mathfrak{g}$ and $W$ is the corresponding Lagrangian subalgebra in the double transversal to $L$. We first note that any $\sigma(u)\in \mathrm{Aut}_{\mathbb{C}[u]}(\mathfrak{g}[u])$ induces an automorphism $\tilde{\sigma}(u)$ of $\mathfrak{g}((u))\oplus \mathfrak{g}$ defined by 
$\tilde{\sigma}(u)=\sigma(u)\oplus \sigma(0)$.
Then, similarly to Case I, one can show that there exists $\sigma(u)\in \mathrm{Aut}_{\mathbb{C}[u]}(\mathfrak{g}[u])$ 
satisfying $\tilde{\sigma}(u)(W\cdot\mathbb{C}[[u^{-1}]])\subseteq \mathbb{O}_{\alpha}\oplus\mathfrak{g}$, where $\alpha$ is either a simple root or $-\alpha_{\max}$. Moreover the following result holds: 

\begin{thm} \cite{SZ}\label{condCaseII}
 Let $\alpha_{\max}=\sum  k_j\alpha_j$, $\alpha_j\in\Gamma$ and  
$\sigma\in\mathrm{Aut}_{\mathbb{C}[u]}(\mathfrak{g}[u])$.

(1) Assume that $\tilde{\sigma}(u)(W\cdot\mathbb{C}[[u^{-1}]])\subseteq \mathbb{O}_{-\alpha_{\max}}\oplus\mathfrak{g}$. Then $\frac{1}{a(u)}$ is a polynomial of degree at most 1. 

(2) Assume that 
$\tilde{\sigma}(u)(W\cdot\mathbb{C}[[u^{-1}]])\subseteq \mathbb{O}_{\alpha_i}\oplus\mathfrak{g}$, for some $i$. Then $\frac{1}{a(u)}$ is a polynomial of degree at most 1, if $k_i=1$ and 
$\frac{1}{a(u)}$ is a constant if $k_i>1$.
\end{thm}
 
\vspace*{0.5cm}

\textbf{Case III.} 
Let us suppose $\bar{\delta}$ satisfies the condition 
$D_{\bar{\delta}}(L)=\mathfrak{g}((u))\oplus (\mathfrak{g}+\varepsilon\mathfrak{g})$, where $\varepsilon^2=0$, and $W$ is the corresponding Lagrangian subalgebra in the double transversal to $L$.

Any $\sigma(u)\in \mathrm{Ad}(\mathfrak{g}[u])$  induces an automorphism $\sigma(0)\in\mathrm{Ad}(\mathfrak{g})$, which in turn gives an automorphism $\bar{\sigma}(0)$ of $\mathfrak{g}+\varepsilon\mathfrak{g}$ defined by $\bar{\sigma}(0)(x+\varepsilon y)=\sigma(0)(x)+\varepsilon\sigma(0)(y)$. Then $\tilde{\sigma}(u)=\sigma(u)\oplus \bar{\sigma}(0)$ is an automorphism of $\mathfrak{g}((u))\oplus (\mathfrak{g}+\varepsilon\mathfrak{g})$. One can prove that there exists $\sigma(u)\in \mathrm{Ad}(\mathfrak{g}[u])$ 
such that $\tilde{\sigma}(u)(W\cdot\mathbb{C}[[u^{-1}]])\subseteq \mathbb{O}_{\alpha}\oplus (\mathfrak{g}+\varepsilon \mathfrak{g})$, where $\alpha$ is either a simple root or $-\alpha_{\max}$. The following result is similar to Theorem \ref{condCaseI} and \ref{condCaseII}:

\begin{thm} \cite{SZ}\label{condCaseIII}
 Let $\alpha_{\max}=\sum  k_j\alpha_j$, $\alpha_j\in\Gamma$ and  
$\sigma\in\mathrm{Ad}(\mathfrak{g}[u])$.

(1) Assume that $\tilde{\sigma}(u)(W\cdot\mathbb{C}[[u^{-1}]])\subseteq \mathbb{O}_{-\alpha_{\max}}\oplus(\mathfrak{g}+\varepsilon \mathfrak{g})$. Then $\frac{1}{a(u)}$ is a constant.

(2) Assume that 
$\tilde{\sigma}(u)(W\cdot\mathbb{C}[[u^{-1}]])\subseteq \mathbb{O}_{\alpha_i}\oplus\mathfrak{g}$, for some $i$. Then $\frac{1}{a(u)}$ is a constant if $k_i=1$, 
and the above inclusion is impossible if $k_i>1$.
\end{thm}

\begin{rem}
In \cite{SZ} it was noticed that, by means of a change of variable in $\mathbb{C}[u]$ and rescaling the nondegenerate bilinear form $Q$, one may assume that $a(u)$ has one of the following forms: 
\begin{enumerate}
\item
$a(u)=1/(1-c_1u)(1-c_2u)$, for non-zero constants $c_1\neq c_2$ 

\item $a(u)=1/(1-u)^2$ 

\item $a(u)=1/1-u$ 

\item $a(u)=1$. 

\end{enumerate}
\end{rem}

\section{Lie bialgebra structures on $\mathfrak{g}[u]$ in Case I}

In this section we will focus on Case I. We first note that the nondegenerate bilinear form on $\mathfrak{g}((u))$ is given by the formula 
\[Q_{a(u)}(f_1(u),f_2(u))=\mathrm{Res}_{u=0}(K(f_1(u),f_2(u))\cdot a(u)),\]
where $K$ is the Killing form of the Lie algebra $\mathfrak{g}((u))$ over 
$\mathbb{C}((u))$.

In \cite{SZ} the following result was proved: 
\begin{prop}
\cite{SZ} There exists a one-to-one correspondence between 
Lie bialgebra structures 
$\delta$ on $\mathfrak{g}[u]$ satisfying  
$D_{\bar{\delta}}(\mathfrak{g}[[u]])=\mathfrak{g}((u))$ and 
bounded Lagrangian subalgebras $W$ of $\mathfrak{g}((u))$, with respect to 
the nondegenerate bilinear form $Q_{a(u)}$, and transversal to $\mathfrak{g}[[u]]$.

\end{prop}

\begin{rem}
If $W$ is a bounded Lagrangian subalgebra of $\mathfrak{g}((u))$ transversal to 
$\mathfrak{g}[[u]]$, then $W\oplus \mathfrak{g}[u]=\mathfrak{g}[u,u^{-1}]$. The equality implies that one can choose dual bases in $W$ and $\mathfrak{g}[u]$ with respect to $Q_{a(u)}$. 
The corresponding Lie bialgebra structure $\delta$
can be reconstructed starting from $W$ in the following way: let us choose a system of Chevalley-Weyl generators $e_{\alpha}$, $e_{-\alpha}$,
$h_{\alpha}$, for all positive roots $\alpha$, such that 
$K(e_{\alpha},e_{-\alpha})=1$ and $h_{\alpha}=[e_{\alpha},e_{-\alpha}]$. The canonical basis of $\mathfrak{g}[u]$ is  formed by $e_{\alpha}u^k$, $e_{-\alpha}u^k$, $h_{\alpha}u^k$, for all positive roots $\alpha$ and all natural $k$. Denote these elements by $e_{\alpha,k}$. 
Let $w_{\alpha,k}$ be a dual basis in $W$ with respect to the the nondegenerate bilinear form $Q_{a(u)}$ and consider the $r$-matrix \[r(u,v)=\sum_{\alpha,k} e_{\alpha,k}\otimes w_{\alpha,k}.\] 
Then \[\delta(f(u))=[f(u)\otimes 1+1\otimes f(v),r(u,v)],\]
for all $f(u)\in \mathfrak{g}[u]$. 

\end{rem}

\begin{prop}\label{max_ord}
Suppose that $W$ is a bounded Lagrangian subalgebra of $\mathfrak{g}((u))$, with respect to $Q_{a(u)}$ and transversal to $\mathfrak{g}[[u]]$. Then there exists $\sigma\in\mathrm{Aut}_{\mathbb{C}[u]}(\mathfrak{g}[u])$ such that $\sigma(u)(W)\subseteq \mathbb{O}_{\alpha}\cap \mathfrak{g}[u,u^{-1}]$, where $\alpha$ is either a simple root or $-\alpha_{\max}$. 
\end{prop}

\begin{proof}
Since $W$ is bounded, we have $W\subseteq \mathfrak{g}[u,u^{-1}]$. 
From \cite{SZ} we also know that there exists $\sigma(u)\in \mathrm{Aut}_{\mathbb{C}[u]}(\mathfrak{g}[u])$ 
such that $\sigma(u)(W\cdot\mathbb{C}[[u^{-1}]])\subseteq \mathbb{O}_{\alpha}$, where $\alpha$ is either a simple root or $-\alpha_{\max}$. 

On the other hand, $\sigma(u)(\mathfrak{g}[u,u^{-1}])=\mathfrak{g}[u,u^{-1}]$. 
One obtains the inclusions: 
$\sigma(u)(W)\subseteq\sigma(u)(W\cdot\mathbb{C}[[u^{-1}]])\cap \sigma(u)(\mathfrak{g}[u,u^{-1}])\subseteq \mathbb{O}_{\alpha}\cap \mathfrak{g}[u,u^{-1}]$. 

\end{proof}

In what follows we will restrict ourselves to the case $\alpha=-\alpha_{\max}$. Let us make the remark that $\mathbb{O}_{-\alpha_{\max}}\cap\mathfrak{g}[u,u^{-1}]=\mathfrak{g}[u^{-1}]$. 
Consider also the Lie algebra $\mathfrak{g}\oplus\mathfrak{g}$, together with
the nondegenerate bilinear form 
\[\bar{Q}((x_1,y_1),(x_2,y_2))=K(x_1,x_2)-K(y_1,y_2).\] 

\begin{prop}
Let $a(u)=1/(1-c_1u)(1-c_2u)$, for non-zero constants $c_1\neq c_2$. There exists a one-to-one correspondence between Lagrangian subalgebras $W$ of $\mathfrak{g}((u))$, with respect to $Q_{a(u)}$, which are transversal to $\mathfrak{g}[[u]]$ and satisfy 
$W \subseteq \mathfrak{g}[u^{-1}]$, and Lagrangian subalgebras in $\mathfrak{g}\oplus\mathfrak{g}$, with respect to $\bar{Q}$, transversal to $\mathrm{diag}(\mathfrak{g})$. 

\end{prop}

\begin{proof}
Assume $W$ is a Lagrangian subalgebra of $\mathfrak{g}((u))$ which is transversal to $\mathfrak{g}[[u]]$ and such that $W \subseteq \mathfrak{g}[u^{-1}]$. This implies that  
$W \supseteq \mathfrak{g}[u^{-1}]^{\perp}=(1-c_1u)(1-c_2u)u^{-2}\mathfrak{g}[u^{-1}]=
(u^{-1}-c_1)(u^{-1}-c_2)\mathfrak{g}[u^{-1}]$. 

The quotient $\frac{W}{(u^{-1}-c_1)(u^{-1}-c_2)\mathfrak{g}[u^{-1}]}$ is a subalgebra of $\frac{\mathfrak{g}[u^{-1}]}{(u^{-1}-c_1)(u^{-1}-c_2)\mathfrak{g}[u^{-1}]}$ which can be identified with $\mathfrak{g}\oplus\mathfrak{g}$. Indeed, let us consider the epimorphism $\psi: \mathfrak{g}[u^{-1}]\longrightarrow \mathfrak{g}\oplus\mathfrak{g}$ defined by $\psi(x)=(x,x)$, for any $x\in \mathfrak{g}$ and $\psi(xu^{-1})=(xc_1,xc_2)$. Then $\mathrm{Ker}(\psi)=(u^{-1}-c_1)(u^{-1}-c_2)\mathfrak{g}[u^{-1}]$, which implies that $\frac{\mathfrak{g}[u^{-1}]}{(u^{-1}-c_1)(u^{-1}-c_2)\mathfrak{g}[u^{-1}]}$ is isomorphic to $\mathfrak{g}\oplus\mathfrak{g}$ via an isomorphism $\hat{\psi}$ induced by $\psi$. 

Let  $\bar{W}$ be the image of $W$ in $\mathfrak{g}\oplus\mathfrak{g}$. Since $W$ is a Lagrangian subalgebra of $\mathfrak{g}((u))$, $\bar{W}$ is a Lagrangian subalgebra of $\mathfrak{g}\oplus\mathfrak{g}$. Moreover, $\mathfrak{g}[[u]]$
projects onto $\mathrm{diag}(\mathfrak{g})$. Then $\bar{W}$ is transversal to $\mathrm{diag}(\mathfrak{g})$ since $W$ is transversal to $\mathfrak{g}[[u]]$. One can easily check that the correspondence $W\mapsto\bar{W}$ is bijective. This ends the proof.

\end{proof}

\begin{cor}\label{sol_I}
Let $c_1$, $c_2$ be different and non-zero complex constants, and 
\[r_{c_1,c_2}=\sum_{\alpha>0}(c_1e_{-\alpha}\otimes e_{\alpha}+c_2e_{\alpha}\otimes e_{-\alpha}+\frac{c_1+c_2}{4}h_{\alpha}\otimes h_{\alpha}),\]
 
\[r_0(u,v)= \frac{1-(c_1+c_2)u+c_1c_2uv}{v-u}\Omega-r_{c_1,c_2}.\]
Then $r_0(u,v)$ provides a Lie bialgebra structure on $\mathfrak{g}[u]$. 
\end{cor}

\begin{proof}
Let $\bar{W_0}$ be the Lie subalgebra of $\mathfrak{g}\oplus\mathfrak{g}$ spanned by the pairs $(e_{-\alpha},0)$, $(0,e_{\alpha})$, $(h_{\alpha},-h_{\alpha})$, for all positive roots $\alpha$. This is obviously a Lagrangian subalgebra of $\mathfrak{g}\oplus\mathfrak{g}$ complementary to the diagonal. 
Then the corresponding Lagrangian subalgebra $W_0$ of $\mathfrak{g}((u))$ is spanned by $\psi^{-1}(e_{-\alpha},0)$, $\psi^{-1}(0,e_{\alpha})$, $\psi^{-1}(h_{\alpha},-h_{\alpha})$ and $(u^{-1}-c_1)(u^{-1}-c_2)\mathfrak{g}[u^{-1}]$. We have: $\psi^{-1}(e_{-\alpha},0)=\frac{(u^{-1}-c_2)e_{-\alpha}}{c_1-c_2}$, $\psi^{-1}(0,e_{\alpha})=\frac{(u^{-1}-c_1)e_{\alpha}}{c_2-c_1}$, $\psi^{-1}(h_{\alpha},-h_{\alpha})=\frac{(2u^{-1}-c_1-c_2)h_{\alpha}}{c_1-c_2}$. 

 Let us choose the following basis in $W_0$: $v^{-k}(v^{-1}-c_1)(v^{-1}-c_2)e_{-\alpha}$, $v^{-k}(v^{-1}-c_1)(v^{-1}-c_2)e_{\alpha}$, 
$\frac{1}{2}v^{-k}(v^{-1}-c_1)(v^{-1}-c_2)h_{\alpha}$, for all $k\geq 0$, $e_{\alpha}(v^{-1}-c_1)$, $e_{-\alpha}(v^{-1}-c_2)$, $\frac{1}{2}h_{\alpha}(v^{-1}-\frac{c_1+c_2}{2})$. 
The corresponding dual elements in $\mathfrak{g}[u]$ are respectively $e_{\alpha}u^{k+1}$, $e_{-\alpha}u^{k+1}$, $h_{\alpha}u^{k+1}$, $e_{-\alpha}$,  $e_{\alpha}$, $h_{\alpha}$. 

One can check by a straightforward computation that the $r$-matrix constructed 
from these dual bases is precisely $r_0(u,v)$. 

\end{proof}

Let us note that $r_0(u,v)$ can be rewritten in the following form: 
\[r_0(u,v)=\frac{1-c_1v-c_2u+c_1c_2uv}{v-u}\Omega+(c_1-c_2)r_{\mathrm{DJ}},\] where $r_{\mathrm{DJ}}=\frac{1}{2}(\sum_{\alpha>0}e_{\alpha}\wedge e_{-\alpha}+\Omega)$. 

\begin{thm}
Let $a(u)=1/(1-c_1u)(1-c_2u)$, where $c_1$ and $c_2$ are different non-zero complex numbers.
Let $r(u,v)$ be an $r$-matrix which corresponds to a Lagrangian subalgebra 
$W$ of $\mathfrak{g}((u))$, with respect to the form $Q_{a(u)}$, such that $W \subseteq \mathfrak{g}[u^{-1}]$. Then
 \[r(u,v)=\frac{1-c_1v-c_2u+c_1c_2uv}{v-u}\Omega+(c_1-c_2)r,\]
where $r\in \mathfrak{g}\otimes\mathfrak{g}$ verifies: 
$r+r^{21}=\Omega$ and $\mathrm{CYB}(r)=0$. 
\end{thm}

\begin{proof}
Since $W$ is a Lagrangian subalgebra of $\mathfrak{g}((u))$ which is transversal to $\mathfrak{g}[[u]]$ and $W \subseteq \mathfrak{g}[u^{-1}]$, it is uniquely defined by a Lagrangian subalgebra $\bar{W}$ of $\mathfrak{g}\oplus \mathfrak{g}$
transversal to $\mathrm{diag}(\mathfrak{g})$. On the other hand, Lagrangian subalgebras with this property are in a one-to-one correspondence with 
solutions of the modified classical Yang--Baxter equation, i.e., $r+r^{21}=\Omega$ and $\mathrm{CYB}(r)=0$ (see \cite{RS, S1}). If $r(u,v)$ corresponds to $W$, then it is uniquely defined by a constant $r$-matrix and the non-constant part 
of $r(u,v)$ is given by the same formula as $r_0(u,v)$. This ends the proof. 
\end{proof}

Consider the Lie algebra $\mathfrak{g}+\varepsilon\mathfrak{g}$ endowed with the
following invariant form: 
$\bar{Q}_{\varepsilon}(x_1+\varepsilon x_2,y_1+\varepsilon y_2)=K(x_1,y_2)+K(x_2,y_1)$. 

\begin{prop}
Let $a(u)=1/(1-u)^2$. There exists a one-to-one correspondence between Lagrangian subalgebras $W$ of $\mathfrak{g}((u))$, with respect to $Q_{a(u)}$, which are transversal to $\mathfrak{g}[[u]]$ and satisfy 
$W \subseteq \mathfrak{g}[u^{-1}]$, and Lagrangian subalgebras in $\mathfrak{g}+\varepsilon\mathfrak{g}$, with respect to $\bar{Q}_{\varepsilon}$, transversal to $\mathfrak{g}$. 
\end{prop}

\begin{proof}
Assume $W$ is a Lagrangian subalgebra of $\mathfrak{g}((u))$ which is transversal to $\mathfrak{g}[[u]]$ and such that $W \subseteq \mathfrak{g}[u^{-1}]$. This implies that  
$W \supseteq \mathfrak{g}[u^{-1}]^{\perp}=(1-u)^2u^{-2}\mathfrak{g}[u^{-1}]=
(u^{-1}-1)^2\mathfrak{g}[u^{-1}]$. 

The quotient $\frac{W}{(u^{-1}-1)^2\mathfrak{g}[u^{-1}]}$ is therefore a
subalgebra of  the Lie algebra $\frac{\mathfrak{g}[u^{-1}]}{(u^{-1}-1)^2\mathfrak{g}[u^{-1}]}$. 

On the other hand the Lie algebra $\frac{\mathfrak{g}[u^{-1}]}{(u^{-1}-1)^2\mathfrak{g}[u^{-1}]}$ can be identified with $\mathfrak{g}+\varepsilon\mathfrak{g}$. Indeed let $\psi:\mathfrak{g}[u^{-1}]\longrightarrow \mathfrak{g}+\varepsilon\mathfrak{g}$ be given by $\psi(x)=x$, $\psi(xu^{-1})=x(1+\varepsilon)$, for all $x\in\mathfrak{g}$. Then $\psi$ is an epimorphism whose kernel equals $(u^{-1}-1)^2\mathfrak{g}[u^{-1}]$. 

The image $\bar{W}$ of $W$ in $\mathfrak{g}+\varepsilon\mathfrak{g}$ is obviously a Lagrangian subalgebra transversal to $\mathfrak{g}$ (we also note that $\mathfrak{g}[[u]]\cap \mathfrak{g}[u^{-1}]=\mathfrak{g}$). One can check that the correspondence which associates $\bar{W}$ to $W$ is bijective. 
\end{proof}

\begin{cor}
Let  
\[r_0(u,v)=\frac{(u-1)(v-1)}{v-u} \Omega.\]
Then $r_0(u,v)$ provides a Lie bialgebra structure on $\mathfrak{g}[u]$. 
\end{cor}

\begin{proof}
Take $\bar{W_0}=\varepsilon\mathfrak{g}$ with canonical basis $\varepsilon e_{\alpha}$, $\varepsilon e_{-\alpha}$, $\varepsilon h_{\alpha}$, for all positive roots $\alpha$. 

Let  $W_0$ be the Lagrangian subalgebra of $\mathfrak{g}((u))$ which corresponds to $\bar{W_0}$. Then $\frac{W_0}{(u^{-1}-1)^2\mathfrak{g}[u^{-1}]}$ is spanned by $(u^{-1}-1)e_{\alpha}$, $(u^{-1}-1)e_{-\alpha}$ and $(u^{-1}-1)h_{\alpha}$. Therefore the Lie algebra $W_0$ is spanned by these elements together with $(u^{-1}-1)^2u^{-k}e_{\alpha}$, $(u^{-1}-1)^2u^{-k}e_{-\alpha}$, $(u^{-1}-1)^2u^{-k}h_{\alpha}$, for all positive roots $\alpha$ and all natural $k$.

The basis in $W_0$ which is dual to the canonical basis in $\mathfrak{g}[u]$ is 
$(u^{-1}-1)e_{\alpha}$, $(u^{-1}-1)e_{-\alpha}$, $(u^{-1}-1)h_{\alpha}$, $(u^{-1}-1)^2u^{-k-2}e_{\alpha}$, $(u^{-1}-1)^2u^{-k-2}e_{-\alpha}$, $(u^{-1}-1)^2u^{-k-2}h_{\alpha}$. An easy computation shows that the $r$-matrix constructed from the dual bases
has the form $r_0(u,v)$. 

\end{proof}

\begin{thm}
Let $a(u)=1/(1-u)^2$ and $r(u,v)$ be an $r$-matrix which corresponds to a Lagrangian subalgebra $W$ of $\mathfrak{g}((u))$, with respect to the form $Q_{a(u)}$, such that $W \subseteq \mathfrak{g}[u^{-1}]$. Then
\[r(u,v)= \frac{(u-1)(v-1)}{v-u} \Omega+r,\]
where $r\in\mathfrak{g}\wedge\mathfrak{g}$ verifies the classical Yang--Baxter equation $\mathrm{CYB}(r)=0$. 
\end{thm}
\begin{proof}
Since $W$ is a Lagrangian subalgebra of $\mathfrak{g}((u))$ which is transversal to $\mathfrak{g}[[u]]$ and $W \subseteq \mathfrak{g}[u^{-1}]$, it is uniquely defined by a Lagrangian subalgebra $\bar{W}$ of 
$\mathfrak{g}+\varepsilon\mathfrak{g}$
transversal to $\mathfrak{g}$. On the other hand, Lagrangian subalgebras $\bar{W}$ with this property are in a one-to-one correspondence with skew-symmetric solutions of the classical Yang--Baxter equation (see \cite{S1}). 

\end{proof}

\begin{prop}
Let $a(u)=1/1-u$. There exists a one-to-one correspondence between Lagrangian subalgebras $W$ of $\mathfrak{g}((u))$, with respect to $Q_{a(u)}$, which are transversal to $\mathfrak{g}[[u]]$ and satisfy 
$W \subseteq \mathfrak{g}[u^{-1}]$, and Lagrangian subalgebras in $\mathfrak{g}\oplus\mathfrak{g}$ transversal to $\mathrm{diag}(\mathfrak{g})$. 

\end{prop}
\begin{proof}
Since $W$ is Lagrangian and contained in $\mathfrak{g}[u^{-1}]$, we have $W \supseteq \mathfrak{g}[u^{-1}]^{\perp}=u^{-2}(1-u)\mathfrak{g}[u^{-1}]=u^{-1}(1-u^{-1})\mathfrak{g}[u^{-1}]$. 

Then $\frac{W}{u^{-1}(1-u^{-1})\mathfrak{g}[u^{-1}]}$ is a subalgebra of $\frac{\mathfrak{g}[u^{-1}]}{u^{-1}(1-u^{-1})\mathfrak{g}[u^{-1}]}$. 

On the other hand, $\frac{\mathfrak{g}[u^{-1}]}{u^{-1}(1-u^{-1})\mathfrak{g}[u^{-1}]}$ is isomorphic to $\mathfrak{g}\oplus\mathfrak{g}$ via a morphism $\psi(x)=(x,x)$, $\psi(xu^{-1})=(0,x)$, for all $x\in\mathfrak{g}$. The projection of  $W$ onto  $\mathfrak{g}\oplus\mathfrak{g}$ becomes a Lagrangian subalgebra which is complementary to the diagonal. 
\end{proof}

\begin{cor}
Let \[r_0(u,v)=\frac{1-u}{v-u} \Omega-r_{DJ},\]
where 
$r_{\mathrm{DJ}}=\frac{1}{2}(\sum_{\alpha>0}e_{\alpha}\wedge e_{-\alpha}+\Omega)$.
Then $r_0(u,v)$ provides a Lie bialgebra structure on $\mathfrak{g}[u]$. 
\end{cor}

\begin{proof}
Let $\bar{W_0}$ be the Lie subalgebra of $\mathfrak{g}\oplus\mathfrak{g}$ spanned by the pairs $(e_{-\alpha},0)$, $(0,e_{\alpha})$, $(h_{\alpha},-h_{\alpha})$, for all positive roots $\alpha$.
Then the corresponding Lagrangian subalgebra $W_0$ of $\mathfrak{g}((u))$ is spanned by the elements $\psi^{-1}(e_{-\alpha},0)=(1-u^{-1})e_{-\alpha}$, $\psi^{-1}(0,e_{\alpha})=u^{-1}e_{\alpha}$, $\psi^{-1}(h_{\alpha},-h_{\alpha})=(1-2u^{-1})h_{\alpha}$ and contains $u^{-1}(1-u^{-1})\mathfrak{g}[u^{-1}]$.

The basis in $W_0$ which is dual to the canonical basis of $\mathfrak{g}[u]$ is the following: $(1-u)u^{-k-1}e_{\alpha}$, $(1-u)u^{-k-1}e_{-\alpha}$, $(1-u)u^{-k-1}h_{\alpha}$, $(u^{-1}-1)e_{-\alpha}$, $u^{-1}e_{\alpha}$, $-\frac{1}{4}(1-2u^{-1})h_{\alpha}$. The corresponding $r$-matrix is $r_0(u,v)$. 

\end{proof}

Consequently, the following result holds: 
\begin{thm}
Let $a(u)=1/1-u$ and $r(u,v)$ be an $r$-matrix which corresponds to a Lagrangian subalgebra $W$ of $\mathfrak{g}((u))$, with respect to the form 
$Q_{a(u)}$, such that $W \subseteq \mathfrak{g}[u^{-1}]$. Then
\[r(u,v)= \frac{1-u}{v-u} \Omega-r,\]
where $r\in\mathfrak{g}\otimes\mathfrak{g}$ verifies $r+r^{21}=\Omega$ and $\mathrm{CYB}(r)=0$.

\end{thm}

The case $a(u)=1$ can be treated in a similar manner and the following results can be easily proved: 
\begin{prop}
Let $a(u)=1$. There exists a one-to-one correspondence between Lagrangian subalgebras $W$ of $\mathfrak{g}((u))$, with respect to $Q_{a(u)}$, which are transversal to $\mathfrak{g}[[u]]$ and satisfy 
$W \subseteq \mathfrak{g}[u^{-1}]$, and Lagrangian subalgebras in $\mathfrak{g}+\varepsilon\mathfrak{g}$ transversal to $\mathfrak{g}$. 
\end{prop}

\begin{thm}
Let $a(u)=1$ and $r(u,v)$ be an $r$-matrix which corresponds to a Lagrangian subalgebra $W$ of $\mathfrak{g}((u))$, with respect to the form 
$Q_{a(u)}$, such that $W \subseteq \mathfrak{g}[u^{-1}]$. Then
\[r(u,v)= \frac{\Omega}{v-u}+r,\]
where $r\in\mathfrak{g}\wedge\mathfrak{g}$ verifies the classical Yang--Baxter equation $\mathrm{CYB}(r)=0$. 
\end{thm}

\section{Lie bialgebra structures on $\mathfrak{g}[u]$ in Case II}

We will analyse in more detail the case $D_{\bar{\delta}}(\mathfrak{g}[[u]])=\mathfrak{g}((u))\oplus\mathfrak{g}$. We note that
the double is endowed with the following nondegenerate bilinear form:
\[ Q_{a(u)}(f_1(u)+x_1,f_2(u)+x_2)=\mathrm{Res}_{u=0}(u^{-1}a(u)K(f_1(u),f_2(u)))-K(x_1,x_2),\]
for all $f_1(u), f_2(u)\in\mathfrak{g}((u))$ and $x_1,x_2\in\mathfrak{g}$. 

According to \cite{SZ}, the following statement holds: 
\begin{prop}
There exists a one-to-one correspondence between Lie bialgebra structures 
$\delta$ on $\mathfrak{g}[u]$ satisfying  
$D_{\bar{\delta}}(\mathfrak{g}[[u]])=\mathfrak{g}((u))\oplus\mathfrak{g}$ and 
bounded Lagrangian subalgebras $W$ of $\mathfrak{g}((u))$, with respect to 
the nondegenerate bilinear form $Q_{a(u)}$, and transversal to $\mathfrak{g}[[u]]$.

\end{prop}

\begin{rem}
If $W$ is a bounded Lagrangian subalgebra of 
$\mathfrak{g}((u))\oplus\mathfrak{g}$ transversal to 
$\mathfrak{g}[[u]]$, then $W\oplus\mathfrak{g}[u]=\mathfrak{g}[u,u^{-1}]\oplus\mathfrak{g}$. Then the corresponding $r$-matrix can be constructed by choosing dual bases in $W$ and $\mathfrak{g}[u]$ with respect to $Q_{a(u)}$ and projecting onto $\mathfrak{g}[u,u^{-1}]$. 
\end{rem}

For any $\sigma\in\mathrm{Aut}_{\mathbb{C}[u]}(\mathfrak{g}[u])$, denote by  $\tilde{\sigma}(u)=\sigma(u)\oplus\sigma(0)$, regarded as an automorphism of $\mathfrak{g}((u))\oplus\mathfrak{g}$. Then the following result holds (we omit the proof which is similar to that of Proposition \ref{max_ord}):

\begin{prop}
Suppose that $W$ is a bounded Lagrangian subalgebra of $\mathfrak{g}((u))\oplus\mathfrak{g}$, with respect to $Q_{a(u)}$ and transversal to $\mathfrak{g}[[u]]$. 
Then there exists $\sigma\in\mathrm{Aut}_{\mathbb{C}[u]}(\mathfrak{g}[u])$ such that $\tilde{\sigma}(u)(W)\subseteq (\mathbb{O}_{\alpha}\cap \mathfrak{g}[u,u^{-1}])\oplus\mathfrak{g}$, where $\alpha$ is either a simple root or $-\alpha_{\max}$. 

\end{prop}

From now on we will restrict ourselves to the case $-\alpha_{\max}$, so that we will study Lagrangian subalgebras $W\subseteq\mathfrak{g}[u^{-1}]\oplus\mathfrak{g}$ transversal to $\mathfrak{g}[[u]]$ . Let us recall that such subalgebras exist if only if $1/a(u)$ has degree at most 1. By a change of variable or rescaling the form $Q$, we have two situations: $a(u)=1/1-u$ or $a(u)=1$. 

\begin{prop}
Let $a(u)=1/1-u$. There exists a one-to-one correspondence between Lagrangian subalgebras $W$ of $\mathfrak{g}((u))\oplus\mathfrak{g}$, with respect to $Q_{a(u)}$, which are transversal to $\mathfrak{g}[[u]]$ and satisfy 
$W \subseteq \mathfrak{g}[u^{-1}]\oplus\mathfrak{g}$, and Lagrangian subalgebras in $\mathfrak{g}\oplus\mathfrak{g}$ transversal to $\mathrm{diag}(\mathfrak{g})$. 
\end{prop}
\begin{proof}
We immediately note that $W$ must contain $(\mathfrak{g}[u^{-1}]\oplus\mathfrak{g})^{\perp}=(1-u)u^{-1}\mathfrak{g}[u^{-1}]=(u^{-1}-1)\mathfrak{g}[u^{-1}]$.
The quotient $\frac{W}{(u^{-1}-1)\mathfrak{g}[u^{-1}]}$ is a subalgebra of 
$\frac{\mathfrak{g}[u^{-1}]}{(u^{-1}-1)\mathfrak{g}[u^{-1}]}\oplus\mathfrak{g}$,
obviously identified with $\mathfrak{g}\oplus\mathfrak{g}$. The conclusion follows by arguments similar to those in Section 2. 
\end{proof}
Consequently, one obtains the following result (whose proof we omit, being similar to those in Section 2): 
\begin{thm}\label{sol_II}
Let $a(u)=1/1-u$ and $r(u,v)$ be an $r$-matrix which corresponds to a Lagrangian subalgebra $W$ of $\mathfrak{g}((u))\oplus\mathfrak{g}$,with respect to the form $Q_{a(u)}$, such that $W \subseteq \mathfrak{g}[u^{-1}]\oplus\mathfrak{g}$. Then
\[r(u,v)= \frac{u(1-v)}{v-u} \Omega+r,\]
where $r\in\mathfrak{g}\otimes\mathfrak{g}$ verifies $r+r^{21}=\Omega$ and 
$\mathrm{CYB}(r)=0$. 
\end{thm}

The remaining case $a(u)=1$ can be treated analogously. 
\begin{prop}
Let $a(u)=1$. There exists a one-to-one correspondence between Lagrangian subalgebras $W$ of $\mathfrak{g}((u))\oplus\mathfrak{g}$, with respect to $Q_{a(u)}$, which are transversal to $\mathfrak{g}[[u]]$ and satisfy 
$W \subseteq \mathfrak{g}[u^{-1}]\oplus\mathfrak{g}$, and Lagrangian subalgebras in $\mathfrak{g}\oplus\mathfrak{g}$ transversal to $\mathrm{diag}(\mathfrak{g})$. 
\end{prop}
\begin{proof}
We note that $W$ must contain $u^{-1}\mathfrak{g}[u^{-1}]$ and $\frac{W}{u^{-1}\mathfrak{g}[u^{-1}]}$ is a subalgebra of $\frac{\mathfrak{g}[u^{-1}]}{u^{-1}\mathfrak{g}[u^{-1}]}\oplus\mathfrak{g}$, which is isomorphic to $\mathfrak{g}\oplus\mathfrak{g}$. Conclusion follows easily.
\end{proof}
 
\begin{thm}
Let $a(u)=1$ and $r(u,v)$ be an $r$-matrix which corresponds to a Lagrangian subalgebra $W$ of $\mathfrak{g}((u))\oplus\mathfrak{g}$,with respect to the form $Q_{a(u)}$, such that $W \subseteq \mathfrak{g}[u^{-1}]\oplus\mathfrak{g}$. Then
\[r(u,v)= \frac{v}{v-u} \Omega+r,\]
where $r\in\mathfrak{g}\otimes\mathfrak{g}$ verifies $r+r^{21}=\Omega$ and 
$\mathrm{CYB}(r)=0$. 

\end{thm}

\begin{rem}
$r$-matrices of the form $\frac{v}{v-u} \Omega+p(u,v)$, where $p(u,v)\in\mathfrak{g}[u,v]$, are called \textit{quasi-trigonometric} and were classified in 
\cite{KPSST, PS}. 
\end{rem}

 \section{Lie bialgebra structures on $\mathfrak{g}[u]$ in Case III}

Finally, let us treat the case $D_{\bar{\delta}}(\mathfrak{g}[[u]])=\mathfrak{g}((u))\oplus(\mathfrak{g}+\varepsilon\mathfrak{g})$.
with the nondegenerate bilinear form on the double given by the formula
\[Q_{a(u)}(f_1(u)+x_2+\varepsilon x_3,f_2(u)+y_2+\varepsilon y_3)=\mathrm{Res}_{u=0}(u^{-2}a(u)K(f_1(u),f_2(u))-\]
\[-K(x_3,y_2)-K(x_2,y_3),\]
for any $f_1(u),f_2(u)\in\mathfrak{g}((u))$ and $x_2,x_3,y_2,y_3\in\mathfrak{g}$.

According to \cite{SZ}, the following result holds:
\begin{prop}
There exists a one-to-one correspondence between Lie bialgebra structures 
$\delta$ on $\mathfrak{g}[u]$ satisfying  
$D_{\bar{\delta}}(\mathfrak{g}[[u]])=\mathfrak{g}((u))\oplus(\mathfrak{g}+\varepsilon\mathfrak{g})$ and 
bounded Lagrangian subalgebras $W$ of $\mathfrak{g}((u))\oplus(\mathfrak{g}+\varepsilon\mathfrak{g})$, with respect to 
the nondegenerate bilinear form $Q_{a(u)}$, and transversal to $\mathfrak{g}[[u]]$.
\end{prop}

\begin{rem}
If $W$ is Lagrangian in $\mathfrak{g}((u))\oplus(\mathfrak{g}+\varepsilon\mathfrak{g})$ and transversal to $\mathfrak{g}[[u]]$, then 
$W\oplus\mathfrak{g}[u]=\mathfrak{g}[u,u^{-1}]\oplus((\mathfrak{g}+\varepsilon\mathfrak{g})$. The corresponding $r$-matrix can be found by choosing dual bases in $W$ and $\mathfrak{g}[u]$ respectively and projecting onto 
$\mathfrak{g}[u,u^{-1}]$. 

\end{rem}
Recall that any $\sigma(u)\in \mathrm{Ad}(\mathfrak{g}[u])$  induces an automorphism $\sigma(0)\in\mathrm{Ad}(\mathfrak{g})$, which in turn gives an well-defined automorphism $\bar{\sigma}(0)$ of $\mathfrak{g}+\varepsilon\mathfrak{g}$ via $\bar{\sigma}(0)(x+\varepsilon y)=\sigma(0)(x)+\varepsilon\sigma(0)(y)$. Then $\tilde{\sigma}(u)=\sigma(u)\oplus \bar{\sigma}(0)$ is an automorphism of $\mathfrak{g}((u))\oplus (\mathfrak{g}+\varepsilon\mathfrak{g})$. 
The following proposition can be proved similarly to Proposition \ref{max_ord}: 
\begin{prop}
Suppose that $W$ is a bounded Lagrangian subalgebra of $\mathfrak{g}((u))\oplus(\mathfrak{g}+\varepsilon\mathfrak{g})$, with respect to $Q_{a(u)}$ and transversal to $\mathfrak{g}[[u]]$. 
Then there exists $\sigma\in\mathrm{Ad}_{\mathbb{C}[u]}(\mathfrak{g}[u])$ such that $\tilde{\sigma}(u)(W)\subseteq (\mathbb{O}_{\alpha}\cap \mathfrak{g}[u,u^{-1}])\oplus(\mathfrak{g}+\varepsilon\mathfrak{g})$, where $\alpha$ is either a simple root or $-\alpha_{\max}$. 

\end{prop}

Such subalgebras exist only when $\alpha$ has coefficient one in the decomposition of $\alpha_{\max}$ or is  $-\alpha_{\max}$. Moreover $a(u)$ should be a constant. Without loss of generality, one may assume that $a(u)=1$.

\begin{prop}
Let $a(u)=1$. There exists a one-to-one correspondence between Lagrangian subalgebras $W$ of $\mathfrak{g}((u))\oplus(\mathfrak{g}+\varepsilon\mathfrak{g})$, with respect to $Q_{a(u)}$, which are transversal to $\mathfrak{g}[[u]]$ and satisfy 
$W \subseteq \mathfrak{g}[u^{-1}]\oplus(\mathfrak{g}+\varepsilon\mathfrak{g})$, and Lagrangian subalgebras in $\mathfrak{g}+\varepsilon\mathfrak{g}$ transversal to $\mathfrak{g}$. 
\end{prop}
\begin{proof}
Being isotropic, $W$ must include $(\mathfrak{g}+\varepsilon\mathfrak{g})^{\perp}=\mathfrak{g}[u^{-1}]$. The quotient $W/\mathfrak{g}[u^{-1}]$ is a Lagrangian subalgebra in $\mathfrak{g}+\varepsilon\mathfrak{g}$ complementary to $\mathfrak{g}$. 
\end{proof}
Consequently, we obtain the description of the corresponding $r$-matrices:
\begin{thm}
Let $a(u)=1$ and $r(u,v)$ be an $r$-matrix which corresponds to a Lagrangian subalgebra $W$ of $\mathfrak{g}((u))\oplus(\mathfrak{g}+\varepsilon\mathfrak{g})$,
with respect to the form 
$Q_{a(u)}$, such that $W \subseteq \mathfrak{g}[u^{-1}]\oplus(\mathfrak{g}+\varepsilon\mathfrak{g})$. Then
\[r(u,v)= \frac{uv}{v-u} \Omega+r,\]
where $r\in\mathfrak{g}\wedge\mathfrak{g}$ verifies $\mathrm{CYB}(r)=0$. 

\end{thm}
\begin{rem}
$r$-matrices of the form $\frac{uv}{v-u} \Omega+p(u,v)$, where
$p(u,v)\in \mathfrak{g}[u,v]$, are called \textit{quasi-rational} and were
studied in \cite{SY}.
\end{rem}

\section{Quasi-twist equivalence between Lie bialgebra structures}

We first remind the reader the notion of twist equivalence between Lie bialgebra structures, according to \cite{KS}:
\begin{defn}
Two Lie bialgebra structures $\delta_1$ and $\delta_2$ on a Lie algebra $L$ are 
called \emph{twist-equivalent} is there exists a Lie algebra isomorphism 
$f: D_{\delta_1}(L)\longrightarrow D_{\delta_2}(L)$ satisfying the following properties:

(1) $Q_1(x,y)=Q_2(f(x),f(y))$, for any $x, y\in D_{\delta_1}(L)$, where $Q_i$ denotes the canonical form on $D_{\delta_i}(L)$, $i=1,2$. 

(2) $f\circ j_1=j_2$, where $j_i$ is the canonical embedding of $L$ into $D_{\delta_i}(L)$.
\end{defn}

Secondly, let us also recall the notion of quantum twisting for
Hopf algebras (see \cite{ES}). Let $A:=A(m,\Delta,\epsilon,S$) be a Hopf algebra with multiplication $m:A\otimes A\rightarrow A$, coproduct $\Delta:A \rightarrow A\otimes A$, counit $\epsilon:A\rightarrow\mathbb C$, and antipode $S:A\to A$. 

An invertible element $F\in A\otimes A$,
$F=\sum_i f^{(1)}_i\otimes f^{(2)}_i$ is called a \emph{quantum twist} if it satisfies
the cocycle equation
\[
F^{12}(\Delta\otimes{\rm id})(F)=F^{23}({\rm id}\otimes\Delta)(F)\,,
\]
and the "unital" normalization condition
\[
(\epsilon \otimes{\rm id})(F)=({\rm id}\otimes\epsilon )(F)=1\,.
\]

One can now define a twisted Hopf algebra
$A^{(F)}:=A^{(F)}(m,\Delta^{(F)},\epsilon,S^{(F)}$) which has the same multiplication $m$
and the counit mapping $\epsilon$ but the twisted coproduct and antipode
\[
\Delta^{(F)}(a)=F\Delta(a)F^{-1},\quad\;S^{(F)}(a)=u\,S(a)u^{-1}, \quad\;u= \sum_i
f^{(1)}_{i}S(f^{(2)}_i).\]

Regarding the quantization of twist-equivalent Lie bialgebra structures, the following result was proved by Halbout in \cite{H}: 

Suppose $\delta_1$ and $\delta_2$ are twist-equivalent and let $(A_1,\Delta_1)$ be a quantization of $(L, \delta_1)$. Then there exists a quantization 
$(A_2,\Delta_2)$ of $(L, \delta_2)$ such that $A_2$ is obtained from $A_1$ via a quantum twist. 

\begin{rem}
Let $\delta_1$ and $\delta_2$ be two Lie bialgebra structures on $L$. 
Assume that there exists an automorphism $\sigma$ of $L$ such that, 
for any $a\in L$, $\delta_2(a)=(\sigma\otimes \sigma)(\delta_1(\sigma^{-1}(a)))$.  Then $\sigma$ is not necessarily extendable to an isomorphism $\bar{\sigma}: D_{\delta_1}(L)\longrightarrow D_{\delta_2}(L)$. In fact, it might happen that the doubles are not isomorphic even as Lie algebras. Let us consider an example. Recall that the $r$-matrix 

\[r(u,v)=\frac{1-uv}{v-u}\Omega+\sum_{\alpha>0} e_{\alpha}\wedge e_{-\alpha}\] 
induces a Lie bialgebra structure on $\mathfrak{g}[u]$ for which the classical double is $(\mathfrak{g}[u,u^{-1}],Q_{1/(1-u^2)})$. 

Let us make the change of variable $2u_1=u+1$, $2v_1=v+1$, which is an automorphism of $\mathfrak{g}[u]$. Then 
\[r(u,v)=2(\frac{u_1(1-v_1)}{v_1-u_1}\Omega+r_{DJ}),\] 
where $r_{DJ}=(r_0+\Omega)/2$. One can notice that 
$r(u,v)$ is proportional to the solution obtained in Theorem \ref{sol_II}. However, this 
solution corresponds to a Lie bialgebra structure for which the double is 
$\mathfrak{g}[u,u^{-1}]\oplus\mathfrak{g}$. Obviously the Lie algebras $\mathfrak{g}[u,u^{-1}]$ and $\mathfrak{g}[u,u^{-1}]\oplus\mathfrak{g}$ are not isomorphic.
\end{rem}

\begin{rem}
$\sigma:\mathfrak{g}[u]\longrightarrow\mathfrak{g}[u]$ given by $\sigma(u)=pu+q$
cannot be extended to $\mathfrak{g}[u,u^{-1}]$. 
\end{rem}

\begin{defn}
Two Lie bialgebra structures $\delta_1$ and $\delta_2$ on $\mathfrak{g}[u]$ are called \emph{quasi-twist equivalent} if $\delta_2(a)=(\sigma\otimes \sigma)(\delta_1(\sigma^{-1}(a)))$, for any $a\in \mathfrak{g}[u]$. Equivalently, the corresponding $r$-matrices satisfy the relation 
$r_2(u,v)=C\cdot r_1(\sigma(u),\sigma(v))$, for some constant $C$. 
\end{defn}
Recall that in Section 2 we studied the Lie bialgebra structures
on $\mathfrak{g}[u]$ whose double is $\mathfrak{g}[u,u^{-1}]$ endowed with
the form $Q_{a(u)}$, where $a(u)=1/(1-c_1u)(1-c_2u)$, for any non-zero constants $c_1\neq c_2$.

By Corollary \ref{sol_I}, the double $(\mathfrak{g}[u,u^{-1}],Q_{1/(1-c_1u)(1-c_2u)})$ leads to the following $r$-matrix: 
 
\[r_{c_1,c_2}(u,v)= \frac{1-(c_1+c_2)u+c_1c_2uv}{v-u}\Omega-r_{c_1,c_2},\]
where 
\[r_{c_1,c_2}=\sum_{\alpha>0}(c_1e_{-\alpha}\otimes e_{\alpha}+c_2e_{\alpha}\otimes e_{-\alpha}+\frac{c_1+c_2}{4}h_{\alpha}\otimes h_{\alpha}).\]

\begin{prop}
For any non-zero complex constants $c_1\neq c_2$, $d_1\neq d_2$, the Lie bialgebra structures with corresponding $r$-matrices $r_{c_1,c_2}(u,v)$ and $r_{d_1,d_2}(u,v)$, are quasi-twist equivalent.
\end{prop}
\begin{proof}
Let us first notice that there exist unique $p$, $q$ such that $d_1=\frac{c_1p}{1-c_1q}$ and $d_2=\frac{c_2p}{1-c_2q}$. Since we also have 

\[r_{d_1,d_2}(u,v)= \frac{1-(d_1+d_2)u+d_1d_2uv}{v-u}\Omega-r_{d_1,d_2},\]
a straightforward computation gives 

\[r_{d_1,d_2}(u,v)(1-c_1q)(1-c_2q)=\frac{1-(c_1+c_2)(pu+q)+c_1c_2(pu+q)(pv+q)}{v-u}\Omega\]\[-p\cdot r_{c_1,c_2}.\]

Thus \[r_{d_1,d_2}(u,v)=\frac{p}{(1-c_1q)(1-c_2q)}\cdot r_{c_1,c_2}(pu+q,pv+q).\]

We have obtained that the Lie bialgebra structures on $\mathfrak{g}[u]$ corresponding to $r_{c_1,c_2}$ and $r_{d_1,d_2}$ are quasi-twist equivalent. 
\end{proof}

Concerning the quantization of the quasi-twist equivalent structures in the example above, the following conjecture was proposed by A. Stolin: 

\textbf{Conjecture.}
Quantizations of Lie bialgebra structures on $\mathfrak{g}[u]$ defined by 
$r_{c_1,c_2}(u,v)$ can be chosen isomorphic as quasi-Hopf algebras. 
 
Here we recall that two Hopf algebras $(A,\Delta_1)$ and $(A,\Delta_2)$ are isomorphic as quasi-Hopf algebras if there is an invertible element $F\in A\otimes A$ and some $A$-invariant element $\Phi$ of $A^{\otimes3}$ such that $F^{12}(\Delta\otimes{\rm id})(F)=F^{23}({\rm id}\otimes\Delta)(F)\cdot\Phi$ and $\Delta_2=F\Delta_1F^{-1}$. 

\vspace*{0.5cm}

\textbf{Acknowledgment.} The authors are thankful to Professor A. Stolin 
for fruitful discussions.

\bibliographystyle{amsalpha}

\end{document}